\documentclass[12pt]{amsart}
\usepackage{fullpage,url,amssymb,enumerate,colonequals}
\usepackage{mathrsfs} 

\usepackage{color}

\newcommand{\defi}[1]{\textsf{#1}} 

\newcommand{\Aff}{\mathbb{A}}

\newcommand{\F}{\mathbb{F}}

\newcommand{\PP}{\mathbb{P}}
\newcommand{\Q}{\mathbb{Q}}
\newcommand{\R}{\mathbb{R}}
\newcommand{\Z}{\mathbb{Z}}

\newcommand{\calA}{\mathcal{A}}

\newcommand{\calS}{\mathcal{S}}

\newcommand{\calX}{\mathcal{X}}

\DeclareMathOperator{\id}{id}

\DeclareMathOperator{\PrePer}{PrePer}

\DeclareMathOperator{\Res}{Res}

\DeclareMathOperator{\Spec}{Spec}

\newcommand{\tors}{{\operatorname{tors}}}

\newcommand{\injects}{\hookrightarrow}

\newtheorem{theorem}{Theorem}[section]

\newtheorem{proposition}[theorem]{Proposition}

\theoremstyle{definition}

\newtheorem{question}[theorem]{Question}

\theoremstyle{remark}
\newtheorem{remark}[theorem]{Remark}

\usepackage[
	backref,
	pdfauthor={Bjorn Poonen}, 
]{hyperref}
\usepackage[alphabetic,backrefs,lite]{amsrefs} 

\begin{document}

\title{Uniform boundedness of rational points and preperiodic points}
\subjclass[2010]{Primary 11G35; Secondary 37P15}
\keywords{Rational points, preperiodic points, uniform boundedness, Morton--Silverman conjecture}
\author{Bjorn Poonen}
\thanks{This research was supported by the Guggenheim Foundation and National Science Foundation grants DMS-0841321 and DMS-1069236.}
\address{Department of Mathematics, Massachusetts Institute of Technology, Cambridge, MA 02139-4307, USA}
\email{poonen@math.mit.edu}
\urladdr{\url{http://math.mit.edu/~poonen/}}
\date{July 2, 2012}

\begin{abstract}
We ask questions generalizing 
uniform versions of conjectures of Mordell and Lang
and combining them with the Morton--Silverman conjecture on preperiodic points.
We prove a few results relating different versions of such questions.
\end{abstract}

\maketitle

\section{Rational points}\label{S:rational}

\subsection{Uniform boundedness questions}

Our goal is to pose some questions about variation of 
the number of rational solutions in a family of polynomial equations.  
The most elementary question of this type we pose is the following:

\begin{question}
\label{Q:4}
For each $n \ge 1$, is there a number $B_n$ such that 
for every $f \in \Q[x_1,\ldots,x_n]$ of total degree~$4$
such that $f(x_1,\ldots,x_n)=0$ has finitely many rational solutions,
the number of solutions is less than or equal to $B_n$?
\end{question}

It will eventually turn out 
that already this question is equivalent to much more general questions,
namely the number field cases of 
Questions \ref{Q:main} and~\ref{Q:main D} 
(see Propositions \ref{P:equivalent} and~\ref{P:reduction to 4}).
In particular, a positive answer for degree~$4$ would imply 
a positive answer for arbitrary degree.

To motivate such questions, 
let us review some uniform boundedness questions in the literature.
If $X$ is a curve of genus $g>1$ over a number field $k$,
then $X(k)$ is finite~\cite{Faltings1983}.
Caporaso, Harris, and Mazur~\cite{Caporaso-Harris-Mazur1997} 
asked whether for each $g>1$ and each $k$, 
there is a constant $B_{g,k}$ such that $\#X(k) \le B_{g,k}$
for all $X$ of genus $g$ over $k$.
The answer is unknown even for $g=2$ and $k=\Q$.
Caporaso, Harris, and Mazur proved that a positive answer
would follow from the Bombieri--Lang conjecture that 
the $k$-rational points on a positive-dimensional variety of general type
are not Zariski dense.
Pacelli~\cite{Pacelli1997} generalized this to show that 
the Bombieri--Lang conjecture implies that the constant $B_{g,k}$ 
can be chosen to depend only on $g$ and $[k:\Q]$;
this would imply its generalization to finitely generated field
extensions of $\Q$
(i.e., function fields of varieties over number fields),
because a curve over a field $k$ of degree $d$ over $\Q(t_1,\ldots,t_n)$
can be specialized to a curve of the same genus 
over a number field of degree $d$ over $\Q$
having at least as many points.
Such results were generalized to higher-dimensional varieties 
for which all subvarieties are of general type:
see~\cites{Abramovich-Voloch1996,Abramovich1997}.

Our main question generalizes such questions 
to arbitrary families of varieties:

\begin{question}
\label{Q:main}
Let $k$ be a finitely generated extension of $\Q$.
Let $\pi \colon X \to S$ be a morphism of finite-type $k$-schemes.
For $s \in X(k)$, let $X_s$ be the fiber $\pi^{-1}(s)$.
Must $\{\#X_s(k) : s \in S(k)\}$ be finite?
\end{question}

If some $\#X_s(k)$ is infinite, that is OK: it contributes 
just the one element $\aleph_0$ to the set whose finiteness is in question.
So the question is really about the uniform boundedness of $\#X_s(k)$
for the $s \in S(k)$ for which $X_s(k)$ is finite.

We generalize further 
by considering points $s$ over finite extensions $L$ of fixed 
(or bounded) degree over $k$:

\begin{question}
\label{Q:main D}
Fix $k$ and $\pi \colon X \to S$ as in Question~\ref{Q:main}.
Let $D \ge 1$. 
Must $\{\#X_s(L) : [L:k] = D, \; s \in S(L)\}$ be finite?
\end{question}

\subsection{Variants}

\begin{question}
\label{Q:Zariski}
Under the hypotheses of Question~\ref{Q:main D}, 
let $z_s \in \Z_{\ge 0}$ be the number of irreducible components of the Zariski
closure of $X_s(L)$ in $X_s$.
Must $\{z_s : [L:k] = D, \; s \in S(L)\}$ be finite?
\end{question}

Given a finite-type $\Q$-scheme $X$, and a subset $A \subset X(\Q)$,
let $\overline{A}$ be the closure of $A$ 
in $X(\R)$ with respect to the Euclidean topology.
Mazur~\cite{Mazur1992} conjectured that the set of connected 
components of the topological space $\overline{X(\Q)}$ is finite
for every $X$.

\begin{question}
\label{Q:Mazur}
Under the hypotheses of Question~\ref{Q:main}, but with $k=\Q$,
let $c_s$ be the number of connected components of $\overline{X_s(\Q)}$.
Must $\{c_s:s \in S(\Q)\}$ be finite?
\end{question}

\subsection{Implications} \label{S:reductions}

Every finite-type $k$-scheme is a finite union of finite-type affine
$k$-schemes,
so each of Questions \ref{Q:main}, \ref{Q:Zariski}, and~\ref{Q:Mazur}
may be reduced to the case where $S$ and $X$ are affine.

Question~\ref{Q:Zariski} is stronger than Question~\ref{Q:main D}.
Question~\ref{Q:Mazur} is stronger than 
the $k=\Q$ case of Question~\ref{Q:main}.
Less trivial is the following:

\begin{proposition}
\label{P:equivalent}
For each finitely generated extension $k$ of $\Q$,
Questions~\ref{Q:main} and~\ref{Q:main D} are equivalent.
\end{proposition}

\begin{proof}
Question~\ref{Q:main} is the $D=1$ case of Question~\ref{Q:main D}.

For the reduction in the opposite direction,
fix an instance $\pi \colon X \to S$ of Question~\ref{Q:main D}.
We may assume that $X$ and $S$ are affine.
View
\[
	T \colonequals \Spec 
	\frac{k[a_{D-1},\ldots,a_0,t]}{(t^D + a_{D-1} t^{D-1} + \cdots a_0)},
\]
as a finite scheme over $\Aff^D = \Spec k[a_{D-1},\ldots,a_0]$.
The restrictions of scalars $\calX \colonequals \Res_{T/\Aff^D}(X \times_k T)$
and $\calS \colonequals \Res_{T/\Aff^D}(S \times_k T)$ 
exist~\cite{Bosch-Lutkebohmert-Raynaud1990}*{7.6, Theorem~4},
and $\pi$ induces $\Pi \colon \calX \to \calS$.
If $a \in \Aff^D(k)$, its fiber in $T$ defines a finite $k$-algebra $L$
(not necessarily a field); then each point $s' \in \calS(k)$ mapping to 
$a$ corresponds to a point $s \in S(L)$, and the fiber $\Pi^{-1}(s')$
equals $\Res_{L/k}(X_s)$, whose $k$-points are in bijection with $X_s(L)$.
Moreover, every degree-$D$ field extension $L$ of $k$ 
arises from some $a \in \Aff^D(k)$.
Thus a positive answer to Question~\ref{Q:main} for $\Pi$
would yield a positive answer to Question~\ref{Q:main D} for $\pi$.
\end{proof}

\begin{proposition}
\label{P:reduction to 4}
Question~\ref{Q:main} for number fields is equivalent to Question~\ref{Q:4}.
\end{proposition}

\begin{proof}
Applying Question~\ref{Q:main} with $k=\Q$ and $X \to S$
the universal family of degree~$4$ hypersurfaces in $\Aff^n$
yields Question~\ref{Q:4}.

Now consider the opposite direction.
The proof of Proposition~\ref{P:equivalent} shows that the 
Question~\ref{Q:main} for number fields 
is equivalent to Question~\ref{Q:main} for $\Q$.
The latter can be reduced to the case where $X$ and $S$ are affine.
To complete the proof, 
we show that each $X_s$ has the same number of rational points
as a certain affine hypersurface of degree~$4$ in $\Aff^n$
for some $n$ depending only on $X \to S$: 
each polynomial in the system defining $X_s$
can be rewritten as a system of equations of degree at most $2$
by Skolem's trick of introducing new indeterminates to represent the results of
intermediate steps in a calculation of a polynomial,
and the union of these systems can be collapsed into a single polynomial
by taking the sum of squares; 
if necessary, add in $z^4$ for a new indeterminate $z$ 
to ensure that the polynomial is of degree exactly~$4$.
The number of indeterminates used in this rewriting of $X_s$
is uniform in $s$.
\end{proof}

\subsection{Counterexamples} \label{S:counterexamples}

Question~\ref{Q:main} has a negative answer for some 
finitely generated fields of characteristic~$p>0$.
For example, if $k \colonequals \F_p(t)$ for some $p>2$,
then in the family of non-smooth curves $X_a \colon x-ax^p=y^p$,
the members with $a \in k-k^p$ have only finitely many $k$-points,
but their number is unbounded 
as $a$ varies~ \cite{Abramovich-Voloch1996}*{Theorem~4.1}.
For another family, this time consisting of smooth curves,
see~\cite{Conceicao-Ulmer-Voloch2012}.

We do not know of ``natural'' fields of characteristic~$0$
for which the answer to Question~\ref{Q:main} is negative,
but we can artificially construct such fields:

\begin{proposition}
\label{P:artificial field}
There exists a countable field $k$ of characteristic~$0$
for which Question~\ref{Q:main} has a negative answer.
\end{proposition}

\begin{proof}
Let $X_1,X_2,\ldots$ be representatives 
for the isomorphism classes of the smooth projective geometrically
integral curves of genus~$2$ over $\Q$.
Let $Y_i \colonequals X_1 \times X_2^2 \times \cdots X_i^i$.
Let $K_i$ be the function field of $Y_i$.
The projection $Y_{i+1} \to Y_i$ induces an injection $K_i \injects K_{i+1}$.
Let $k = \varinjlim K_i$.

For each $i$, composing any of the $i$ projections $Y_i \to X_i$
with an automorphism of $X_i$ yields an element of 
$X_i(Y_i) = X_i(K_i) \subseteq X_i(k)$.
Since any nonconstant morphism
between genus~$2$ curves in characteristic~$0$ is an isomorphism
and the automorphism group of a genus~$2$ curve is finite,
any nonconstant morphism $Y_j \to X_i$ for $j \ge i$ 
factors through one of the projections $Y_j \to X_i$
and hence corresponds to an already-constructed point of $X_i(k)$.

Thus $i \le \# X_i(k) < \infty$ for each $i$,
so Question~\ref{Q:main} for a versal family of genus~$2$ curves over $k$
has a negative answer.
\end{proof}

\section{Torsion points on abelian varieties}\label{S:torsion}

If $A$ is an abelian variety over a number field $k$,
then the torsion subgroup $A(k)_{\tors}$ is finite: 
this is a small part of the Mordell--Weil theorem~\cite{Weil1929},
and can be proved using height functions or $p$-adic methods.
This suggests the following well-known question:

\begin{question}
\label{Q:torsion}
Is there a bound on $\#A(k)_{\tors}$ depending only on $\dim A$ and $[k:\Q]$?
\end{question}

For $\dim A=1$, the answer is 
yes~\cites{Mazur1977,Kamienny-Mazur1995,Merel1996}.
For $\dim A>1$, there are only partial results:
see~\cite{Cadoret-Tamagawa2011}, which also considers the geometric analogue.
If the answer is yes, then the answer is yes also over finitely generated
extensions $k$ of $\Q$:
restriction of scalars lets us reduce to the case $k=\Q(t_1,\ldots,t_n)$,
and then specialization lets us remove one indeterminate at a time
without enlarging the torsion subgroup.

\begin{remark}
Uniform boundedness for the number of rational points on 
curves of genus $g>1$ over a finitely generated extension $k$ of $\Q$
for each $g$ and $k$ would 
imply a positive answer to Question~\ref{Q:main} for all families
$X \to S$ whose fibers are of dimension at most~$1$.
Indeed, the geometry of the singularities and
$0$-dimensional components of the fibers is uniformly bounded,
and irreducible components that are not geometrically irreducible
have rational points constrained to the singular locus,
so it would suffice to prove uniform boundedness of $C(k)$
for smooth projective geometrically integral curves $C$ of bounded genus
with finitely many $k$-points.
If $C$ has genus~$0$, then $C(k)$ is empty or infinite.
If $C$ has genus~$1$, then $C(k)$ is empty, infinite, or
of cardinality bounded by the previous paragraph.
If $C$ has genus greater than $1$, then our hypothesis applies.
\end{remark}

\section{Preperiodic points}\label{S:preperiodic}

Given a morphism $f \colon X \to X$ of $k$-schemes,
a point in $X(k)$ is called \defi{preperiodic} 
if its forward trajectory is finite;
let $\PrePer(f,k)$ be the set of such points.
Northcott~\cite{Northcott1950} invented the theory of height functions
to prove that if $k$ is a number field 
and $f \colon \PP^n \to \PP^n$ is a morphism of degree $d \ge 2$ over $k$,
then $\PrePer(f,k)$ is finite.
The Morton--Silverman conjecture~\cite{Morton-Silverman1994}*{p.~100} 
predicts that $\#\PrePer(f,k)$ is bounded by a constant depending
only on $n$, $d$, and $[k:\Q]$.

\begin{remark}
\label{R:Fakhruddin}
Morton and Silverman observed that applying their conjecture
to the morphism $\PP^1 \to \PP^1$ induced by multiplication-by-$2$
on the $x$-coordinate of an elliptic curve $A$
yields uniform boundedness of torsion points on elliptic curves over number
fields.
By~\cite{Fakhruddin2003}*{Corollary~2.4}, 
the Morton--Silverman conjecture also implies
a positive answer to Question~\ref{Q:torsion} 
for abelian varieties of arbitrary dimension.
\end{remark}

We may now ask the analogue of Question~\ref{Q:main D}
for rational preperiodic points:

\begin{question}
\label{Q:main preperiodic}
Let $k$ be a finitely generated extension of $\Q$.
Let $\pi \colon X \to S$ be a morphism of finite-type $k$-schemes.
Let $f \colon X \to X$ be an $S$-morphism.
If $L$ is a finite extension of $k$ and $s \in S(L)$, 
let $X_s \colonequals \pi^{-1}(s)$
and let $f_s \colon X_s \to X_s$ be the restriction of $f$ to $X_s$.
Let $D \ge 1$.
Must $\{\#\PrePer(f_s,L) : [L:k] = D, \; s \in S(L)\}$ be finite?
\end{question}

\subsection{Variants}

Questions \ref{Q:Zariski} and~\ref{Q:Mazur}
also admit analogues in which $X_s(L)$ is replaced by $\PrePer(f_s,L)$.

\subsection{Implications}

Taking $f=\id$ in Question~\ref{Q:main preperiodic} 
yields Question~\ref{Q:main D}.

Question~\ref{Q:main preperiodic} for the universal
family of degree-$d$ self-maps $\PP^n \to \PP^n$
is equivalent to the Morton--Silverman conjecture,
so by Remark~\ref{R:Fakhruddin}, 
a positive answer to Question~\ref{Q:main preperiodic} would imply 
a positive answer to Question~\ref{Q:torsion}.
In fact, a positive answer to Question~\ref{Q:main preperiodic} 
also implies a positive answer to Question~\ref{Q:torsion} directly:
use Zarhin's trick~\citelist{\cite{Zarhin1974-trick}; \cite{Milne-Abelian1986}*{Remark~16.12}} 
to reduce to the case of principally polarized
abelian varieties of a fixed dimension ($8$ times as large), 
for which a versal family $\calA \to S$ exists,
and then apply Question~\ref{Q:main preperiodic}
to the $S$-morphism $[2] \colon \calA \to \calA$.

As in Section~\ref{S:reductions}, Question~\ref{Q:main preperiodic}
can be reduced to the case in which $S$ is affine,
but it is not clear whether we can assume also that $X$ is affine,
since $X$ might not be a union of \emph{$f$-stable} affine subschemes.
Because of this, 
the analogue of Proposition~\ref{P:equivalent}
for preperiodic points is weakened slightly to ensure that
the restrictions of scalars in its proof exist without first making
$X$ affine:

\begin{proposition}
\label{P:equivalent preperiodic}
Let $k$ be a finitely generated extension of $\Q$.
If the answer to Question~\ref{Q:main preperiodic} 
for quasi-projective schemes over $k$ is positive for $D=1$,
then it is positive also for arbitrary~$D$.
\end{proposition}

\begin{proof}
Let $\pi \colon X \to S$ and $f \colon X \to X$ 
be an instance of Question~\ref{Q:main preperiodic} for a given $k$ and $D$.
We can no longer assume that $X$ is affine,
but since $X$ and $S$ are quasi-projective,
\cite{Bosch-Lutkebohmert-Raynaud1990}*{7.6, Theorem~4}
still applies to let us construct
$\Pi \colon \calX \to \calS$ as in the proof of 
Proposition~\ref{P:equivalent},
and we also obtain an $\calS$-morphism $F \colon \calX \to \calX$.
Each $s' \in \calS(k)$, corresponds to a finite $k$-algebra $L$
with a point $s \in S(L)$,
and $\PrePer(F_{s'},k) \subseteq \calX_{s'}(k)$ 
corresponds to $\PrePer(f_s,L) \subseteq X_s(L)$.
So a positive answer to Question~\ref{Q:main preperiodic} 
for $(\Pi,F,1)$ would yield a positive answer for $(\pi,f,D)$.
\end{proof}

\section*{Acknowledgements} 

I thank Doug Ulmer and Paul Vojta for a comment.


\begin{bibdiv}
\begin{biblist}


\bib{Abramovich1997}{article}{
  author={Abramovich, Dan},
  title={A high fibered power of a family of varieties of general type dominates a variety of general type},
  journal={Invent. Math.},
  volume={128},
  date={1997},
  number={3},
  pages={481--494},
  issn={0020-9910},
  review={\MR {1452430 (98e:14034)}},
  doi={10.1007/s002220050149},
}

\bib{Abramovich-Voloch1996}{article}{
  author={Abramovich, Dan},
  author={Voloch, Jos{\'e} Felipe},
  title={Lang's conjectures, fibered powers, and uniformity},
  journal={New York J. Math.},
  volume={2},
  date={1996},
  pages={20--34, electronic},
  issn={1076-9803},
  review={\MR {1376745 (97e:14031)}},
}

\bib{Bosch-Lutkebohmert-Raynaud1990}{book}{
  author={Bosch, Siegfried},
  author={L{\"u}tkebohmert, Werner},
  author={Raynaud, Michel},
  title={N\'eron models},
  series={Ergebnisse der Mathematik und ihrer Grenzgebiete (3) [Results in Mathematics and Related Areas (3)]},
  volume={21},
  publisher={Springer-Verlag},
  place={Berlin},
  date={1990},
  pages={x+325},
  isbn={3-540-50587-3},
  review={\MR {1045822 (91i:14034)}},
}

\bib{Cadoret-Tamagawa2011}{article}{
  author={Cadoret, Anna},
  author={Tamagawa, Akio},
  title={On a weak variant of the geometric torsion conjecture},
  journal={J. Algebra},
  volume={346},
  date={2011},
  pages={227--247},
  issn={0021-8693},
  review={\MR {2842079}},
  doi={10.1016/j.jalgebra.2011.09.002},
}

\bib{Caporaso-Harris-Mazur1997}{article}{
  author={Caporaso, Lucia},
  author={Harris, Joe},
  author={Mazur, Barry},
  title={Uniformity of rational points},
  journal={J. Amer. Math. Soc.},
  volume={10},
  date={1997},
  number={1},
  pages={1--35},
  issn={0894-0347},
  review={\MR {1325796 (97d:14033)}},
  doi={10.1090/S0894-0347-97-00195-1},
}

\bib{Conceicao-Ulmer-Voloch2012}{article}{
  author={Concei\c {c}\~{a}o, Ricardo},
  author={Ulmer, Douglas},
  author={Voloch, Jos\'e Felipe},
  title={Unboundedness of the number of rational points on curves over function fields},
  journal={New York J. Math.},
  volume={18},
  date={2012},
  pages={291--293},
}

\bib{Fakhruddin2003}{article}{
  author={Fakhruddin, Najmuddin},
  title={Questions on self maps of algebraic varieties},
  journal={J. Ramanujan Math. Soc.},
  volume={18},
  date={2003},
  number={2},
  pages={109--122},
  issn={0970-1249},
  review={\MR {1995861 (2004f:14038)}},
}

\bib{Faltings1983}{article}{
  author={Faltings, G.},
  title={Endlichkeitss\"atze f\"ur abelsche Variet\"aten \"uber Zahlk\"orpern},
  language={German},
  journal={Invent. Math.},
  volume={73},
  date={1983},
  number={3},
  pages={349\ndash 366},
  issn={0020-9910},
  review={\MR {718935 (85g:11026a)}},
  translation={ title={Finiteness theorems for abelian varieties over number fields}, booktitle={Arithmetic geometry (Storrs, Conn., 1984)}, pages={9\ndash 27}, translator = {Edward Shipz}, publisher={Springer}, place={New York}, date={1986}, note={Erratum in: Invent.\ Math.\ {\bf 75} (1984), 381}, },
}

\bib{Kamienny-Mazur1995}{article}{
  author={Kamienny, S.},
  author={Mazur, B.},
  title={Rational torsion of prime order in elliptic curves over number fields},
  note={With an appendix by A. Granville; Columbia University Number Theory Seminar (New York, 1992)},
  journal={Ast\'erisque},
  number={228},
  date={1995},
  pages={3, 81--100},
  issn={0303-1179},
  review={\MR {1330929 (96c:11058)}},
}

\bib{Mazur1977}{article}{
  author={Mazur, B.},
  title={Modular curves and the Eisenstein ideal},
  journal={Inst. Hautes \'Etudes Sci. Publ. Math.},
  number={47},
  date={1977},
  pages={33--186 (1978)},
  issn={0073-8301},
  review={\MR {488287 (80c:14015)}},
}

\bib{Mazur1992}{article}{
  author={Mazur, Barry},
  title={The topology of rational points},
  journal={Experiment. Math.},
  volume={1},
  date={1992},
  number={1},
  pages={35\ndash 45},
  issn={1058-6458},
  review={\MR {1181085 (93j:14020)}},
}

\bib{Merel1996}{article}{
  author={Merel, Lo{\"{\i }}c},
  title={Bornes pour la torsion des courbes elliptiques sur les corps de nombres},
  language={French},
  journal={Invent. Math.},
  volume={124},
  date={1996},
  number={1-3},
  pages={437\ndash 449},
  issn={0020-9910},
  review={\MR {1369424 (96i:11057)}},
}

\bib{Milne-Abelian1986}{article}{
  author={Milne, J. S.},
  title={Abelian varieties},
  conference={ title={Arithmetic geometry}, address={Storrs, Conn.}, date={1984}, },
  book={ publisher={Springer}, place={New York}, },
  date={1986},
  pages={103--150},
  review={\MR {861974}},
}

\bib{Morton-Silverman1994}{article}{
  author={Morton, Patrick},
  author={Silverman, Joseph H.},
  title={Rational periodic points of rational functions},
  journal={Internat. Math. Res. Notices},
  date={1994},
  number={2},
  pages={97--110},
  issn={1073-7928},
  review={\MR {1264933 (95b:11066)}},
  doi={10.1155/S1073792894000127},
}

\bib{Northcott1950}{article}{
  author={Northcott, D. G.},
  title={Periodic points on an algebraic variety},
  journal={Ann. of Math. (2)},
  volume={51},
  date={1950},
  pages={167--177},
  issn={0003-486X},
  review={\MR {0034607 (11,615c)}},
}

\bib{Pacelli1997}{article}{
  author={Pacelli, Patricia L.},
  title={Uniform boundedness for rational points},
  journal={Duke Math. J.},
  volume={88},
  date={1997},
  number={1},
  pages={77--102},
  issn={0012-7094},
  review={\MR {1448017 (98b:14020)}},
  doi={10.1215/S0012-7094-97-08803-7},
}

\bib{Weil1929}{article}{
  author={Weil, Andr{\'e}},
  title={L'arithm\'etique sur les courbes alg\'ebriques},
  language={French},
  journal={Acta Math.},
  volume={52},
  date={1929},
  number={1},
  pages={281--315},
  issn={0001-5962},
  review={\MR {1555278}},
  doi={10.1007/BF02547409},
}

\bib{Zarhin1974-trick}{article}{
  author={Zarhin, Ju. G.},
  title={A remark on endomorphisms of abelian varieties over function fields of finite characteristic},
  language={Russian},
  journal={Izv. Akad. Nauk SSSR Ser. Mat.},
  volume={38},
  date={1974},
  pages={471--474},
  issn={0373-2436},
  review={\MR {0354689 (50 \#7166)}},
}

\end{biblist}
\end{bibdiv}
\end{document}